\documentclass[a4paper, draft, 12pt]{amsproc}
\usepackage{amssymb}
\usepackage{amsmath,amssymb,ulem,color}

\usepackage[hyphens]{url} 
\usepackage[dvips]{graphicx} 
\usepackage{cmdtrack} 
\usepackage{mathrsfs}                       

\theoremstyle{plain}
 \newtheorem{theorem}{Theorem}[section]
 \newtheorem{prop}{Proposition}[section]
 \newtheorem{lem}{Lemma}[section]
 
\theoremstyle{Definition}
 \newtheorem{exm}{Example}[section]
 
 \newtheorem{dfn}{Definition}[section]
\theoremstyle{remark}

 \numberwithin{equation}{section}

\renewcommand{\leq}{\leqslant}
\renewcommand{\geq}{\geqslant}

\title[A Short Comparison of Classical Complex Dynamics and Holomorphic...]{A Short Comparison  of Classical Complex Dynamics and Holomorphic Semigroup Dynamics}



\author[B. H. Subedi]{\bfseries  Bishnu Hari Subedi}

\address{ %% Put here your affiliation; street address is not required
Central Department of Mathematics \\ % \hfill (Received 00 00 201?)\\
Institute of Science and Technology   \\ %\hfill (Revised  00 00 201?)\\
Tribhuvan University   \\ %\hfill (Revised  00 00 201?)\\
Kirtipur, Kathmandu\\
Nepal}
\email{subedi.abs@gmail.com / subedi\_bh@cdmathtu.edu.np }

\author[A. Singh]{Ajaya Singh}
\address{Central Department of Mathematics, Institute of Science and Technology, Tribhuvan University, Kirtipur, Kathmandu, Nepal }
\email{singh.ajaya1@gmail.com / singh\_a@cdmathtu.edu.np} 
\thanks{This research work of the first author is supported by PhD faculty fellowship from University Grants Commission, Nepal. } %% optional
\begin{document}

{\begin{flushleft}\baselineskip9pt\scriptsize
%PUBLICATIONS DE L'INSTITUT MATH\'EMATIQUE\newline
%Nouvelle s\'erie, tome ??(1??)) (201?), od--do \hfill DOI: \\
%MANUSCRIPT
\end{flushleft}}
\vspace{18mm} \setcounter{page}{1} \thispagestyle{empty}

\begin{abstract}
This is an expository plus research paper which mainly exposes preliminary connection and contrast between classical complex dynamics and semigroup dynamics of holomorphic functions. Classically, we expose some existing results of rational and transcendental dynamics and we see how far these results generalized to holomorphic semigroup dynamics as well as we also see what new phenomena occur. 
\end{abstract}

\maketitle

\section{Introduction}
It is quite natural to extend the Fatou-Julia-Eremenko  theory of the iteration of single holomorphic map in complex plane $ \mathbb{C} $ or extended complex plane $ \mathbb{C}_{\infty} $ to composite of the family of holomorphic maps. So, the purpose of this paper is to expose the theory of complex dynamics not only for the iteration of single holomorphic map on $ \mathbb{C} $ or $ \mathbb{C_{\infty}} $ but also for the composite of the family $\mathscr{F}  $ of such maps. Let $ \mathscr{F} $ be a space of holomorphic maps on $ \mathbb{C} $ or $ \mathbb{C}_{\infty} $. For any map $ \phi \in \mathscr{F} $,  $ \mathbb{C} $ or $ \mathbb{C}_{\infty} $ is naturally partitioned  into two subsets: the set of normality and its complement. We say that a family $ \mathscr{F} $ is normal if each sequence from the family has a subsequence which either converges uniformly on compact subsets  or diverges uniformly to $ \infty $. The set of normality or Fatou set $ F(\phi) $ of the map $ \phi \in \mathscr{F} $ is the largest open set on which the iterates $ \phi^{n} = \phi \circ \phi \circ \ldots \circ \phi$ (n-fold composition of $ \phi $ with itself) is a normal family.  The complement $ J(\phi) $ is the Julia set. A maximally connected subset of the Fatou set $ F(\phi) $ is a Fatou component.  The main concern of such an iteration theory is to describe the nature of the components of Fatou set and the structure and properties of the Julia set.   
 
 In our study, classical complex dynamics  refers  the iteration theory of single holomorphic map and holomorphic semigroup dynamics refers the dynamical theory generated by various classes of holomorphic maps. In holomorphic semigroup dynamics, algebraic structure of semigroup naturally attached to the dynamics and hence the situation is largely complicated. The principal aim  of this paper is to see how far classical complex dynamics applies to holomorphic semigroup dynamics and what new phenomena appear in holomorphic semigroup settings. 
 
\section{The notion of holomorphic semigroup}
Semigroup $ S $ is a very classical algebraic structure with a binary composition that satisfies associative law. It naturally arose from the general mapping of a set into itself. So a set of holomorphic maps on $ \mathbb{C} $ or $ \mathbb{C}_{\infty} $ naturally forms a semigroup. Here, we take a set $ A $ of holomorphic maps and construct a semigroup $ S $ consists of all elements that can be expressed as a finite composition of elements in $ A $. We say such a semigroup $ S $ by \textit{holomorphic semigroup} generated by set $ A $.  
For our simplicity, we denote the class of all rational maps on $ \mathbb{C_{\infty}} $ by $ \mathscr{R} $ and class of all transcendental entire maps on $ \mathbb{C} $ by $ \mathscr{E} $. 
Our particular interest is to study of the dynamics of the families of above two classes of holomorphic maps.  For a collection $\mathscr{F} = \{f_{\alpha}\}_{\alpha \in \Delta} $ of such maps, let 
$$
S =\langle f_{\alpha} \rangle
$$ 
be a \textit{holomorphic semigroup} generated by them. Here $ \mathscr{F} $ is either a collection $ \mathscr{R} $ of rational maps or a collection $ \mathscr{E} $ of transcendental entire maps. The index set $ \Delta $ to which $ \alpha $  belongs is allowed to be infinite in general unless otherwise stated. 
Here, each $f \in S$ is a holomorphic function and $S$ is closed under functional composition. Thus, $f \in S$ is constructed through the composition of finite number of functions $f_{\alpha_k},\;  (k=1, 2, 3,\ldots, m) $. That is, $f =f_{\alpha_1}\circ f_{\alpha_2}\circ f_{\alpha_3}\circ \cdots\circ f_{\alpha_m}$. In particular,  if $ f_{\alpha} \in \mathscr{R} $, we say $ S =\langle f_{\alpha} \rangle$ a \textit{rational semigroup} and if  $ f_{\alpha} \in \mathscr{E} $, we say $ S =\langle f_{\alpha} \rangle$ a \textit{transcendental semigroup}. 

A semigroup generated by finitely many holomorphic functions $f_{i}, (i = 1, 2, \ldots, \\ n) $  is called \textit{finitely generated  holomorphic semigroup}. We write $S= \langle f_{1},f_{2},\ldots,f_{n} \rangle$.
 If $S$ is generated by only one holomorphic function $f$, then $S$ is \textit{cyclic semigroup}. We write $S = \langle f\rangle$. In this case, each $g \in S$ can be written as $g = f^n$, where $f^n$ is the nth iterates of $f$ with itself. Note that in our study of  semigroup dynamics, we say $S = \langle f\rangle$  a \textit{trivial semigroup}. 
  
The following result will be clear from the definition of holomorphic semigroup. It shows that every element of holomorphic semigroup can be written as finite composition of the sequence of $f_{\alpha} $
\begin{prop}\label{ts1}
 Let $S   = \langle f_{\alpha} \rangle$ be an arbitrary  holomorphic semigroup. Then for every $ f \in S $,  $f^{m} $(for all $ m \in \mathbb{N}$) can be written as $f^{m} =f_{\alpha_1}\circ f_{\alpha_2}\circ f_{\alpha_3}\circ \cdots\circ f_{\alpha_p}$ for some $ p\in \mathbb{N} $.
 \end{prop}
This proposition \ref{ts1} tells us that  it need not to apply same map over and over again in holomorphic semigroup dynamics. But instead, we may start with family of maps and we consider dynamics over iteratively defined composition  of sequence of maps. 

Next, we define and discuss some special collection and sequences of holomorphic functions.
Note that all notions of convergence that we deal in this paper will be with respect to the Euclidean metric on the complex plane $ \mathbb{C} $ or spherical metric on the Riemann sphere $ \mathbb{C}_{\infty} $.

The family $\mathscr{F}$  of complex analytic maps forms a \textit{normal family} in a domain $ D $ if given any composition sequence $ (f_{\alpha}) $ generated by the member of  $ \mathscr{F} $,  there is a subsequence $( f_{\alpha_{k}}) $ which is uniformly convergent or divergent on all compact subsets of $D$. If there is a neighborhood $ U $ of the point $ z\in\mathbb{C} $ such that $\mathscr{F} $ is normal family in $U$, then we say $ \mathscr{F} $ is normal at $ z $. If  $\mathscr{F}$ is a family of members from the semigroup $ S $, then we simply say that $ S $ is normal in the neighborhood of $ z $ or $ S $ is normal at $ z $.

Let  $ f $ be a holomorphic map. We say that  $ f $ \textit{iteratively divergent} at $ z \in \mathbb{C} $ if $  f^n(z)\rightarrow \alpha \; \textrm{as} \; n \rightarrow \infty$, where $ \alpha $  is an essential singularity of $ f $. A sequence $ (f_{k})_{k \in \mathbb{N}} $ of holomorphic maps is said to be \textit{iteratively divergent} at $ z $ if $ f_{k}^{n}(z) \to\alpha_{k} \;\ \text{as}\;\ n\to \infty$ for all $ k \in \mathbb{N} $, where $ \alpha_{k} $  is an essential singularity of $ f_{k} $ for each $ k $.  Semigroup $ S $ is \textit{iteratively divergent} at $ z $ if $f^n(z)\rightarrow \alpha_{f} \; \textrm{as} \; n \rightarrow \infty$, where $ \alpha_{f} $  is an essential singularity of each $ f \in S $. Otherwise, a function $ f  $, sequence $ (f_{k})_{k \in \mathbb{N}} $ and semigroup $ S $  are said to be \textit{iteratively bounded} at $ z $. 

\section{Fatou set, Julia set and Escaping set}
In classical complex dynamics, each of Fatou set, Julia set and escaping set are defined in two different but equivalent ways.  In first definition, Fatou set is defined as the set of normality of the iterates of given function, Julia set is defined as the complement of the Fatou set and escaping set is defined as the set of points that goes to essential singularity under the iterates of given function. The second definition of  Fatou set is given as a largest completely invariant open set and Julia set is given as a smallest completely invariant close set  whereas escaping set is a completely invariant non-empty neither open nor close set in $\mathbb{C} $. 
Each of these definitions can be naturally extended to the settings of holomorphic semigroup $ S $ but extension definitions are not equivalent.  
Based on above first definition (that is,  on the Fatou-Julia-Eremenko theory of a complex analytic function), the Fatou set, Julia set and escaping set in the settings of holomorphic semigroup are defined as follows.
\begin{dfn}[\textbf{Fatou set, Julia set and escaping set}]\label{2ab} 
\textit{Fatou set} of the holomorphic semigroup $S$ is defined by
  $$
  F (S) = \{z \in \mathbb{C}: S\;\ \textrm{is normal in a neighborhood of}\;\ z\}
  $$
and the \textit{Julia set} $J(S) $ of $S$ is the compliment of $ F(S) $. If $ S $ is a transcendental semigroup, the \textit{escaping set} of $S$ is defined by 
$$
I(S)  = \{z \in \mathbb{C}: S \;  \text{is iteratively divergent at} \;z \}
$$
We call each point of the set $  I(S) $ by \textit{escaping point}.        
\end{dfn} 
It is obvious that $F(S)$ is the largest open subset (of $\mathbb{C}$ or $ \mathbb{C}_{\infty} $) on which the family $\mathscr{F} $ in $S$ (or semigroup $ S $ itself) is normal. Hence its compliment $J(S)$ is a smallest closed set for any  semigroup $S$. Whereas the escaping set $ I(S) $ is neither an open nor a closed set (if it is non-empty) for any semigroup $S$. Any maximally connected subset $ U $ of the Fatou set $ F(S) $ is called a \textit{Fatou component}.  
        
If $S = \langle f\rangle$, then $F(S), J(S)$ and $I(S)$ are respectively the Fatou set, Julia set and escaping set in classical complex dynamics. In this situation we simply write: $F(f), J(f)$ and $I(f)$. 

The main motivation of this paper comes from seminal work of Hinkkanen and Martin \cite{hin} on the dynamics of rational semigroup and the extension study of K. K. Poon \cite{poo} to the dynamics of transcendental semigroup. Both of them naturally generalized classical complex dynamics to the dynamics of the sequence of different functions by means of composition. Another motivation of studying escaping set of transcendental semigroup comes from the work of Dinesh Kumar and Sanjay Kumar\cite{kum2, kum1} where they defined escaping set and discussed how far escaping set of classical transcendental dynamics can be generalized to semigroup dynamics. 

The fundamental contrast between classical complex dynamics and semigroup dynamics appears by different algebraic structure of corresponding semigroups. In fact, non-trivial semigroup (rational or transcendental) need not be, and most often will not be abelian. However, trivial semigroup is cyclic and therefore abelian. As we discussed before, classical complex dynamics is a dynamical study of trivial (cyclic)  semigroup whereas semigroup dynamics is a dynamical study of non-trivial holomorphic semigroup. 

The following immediate result holds good from definition \ref{2ab} of escaping set.
\begin{theorem}\label{1df}
Let $ S $ be a transcendental semigroup and let $ z \in \mathbb{C} $ is an escaping point under $ S $. Then every sequence $ (g_{k})_{k \in \mathbb{N}} $ in $ S $ is iteratively divergent at $ z $  and  every subsequence of $ (g_{k})_{k \in \mathbb{N}} $ is also iteratively divergent at $ z $.
\end{theorem}

The following characterization of escaping set will be clear from the definition \ref{2ab} of escaping set and proposition \ref{ts1}, which can be an alternative definition.

\begin{theorem}\label{ad1}
If a complex number $ z \in \mathbb{C} $ is escaping point of any transcendental semigroup $ S $, then  every sequence in $ S $ has a subsequence which diverges to $ \infty $ at $ z $.
\end{theorem}

On the basis of the theorem \ref{ad1}, we can say that our definition \ref {2ab} of escaping set is more general than that of the definition of Dinesh Kumar and Sanjay Kumar {\cite[Definition 2.1]{kum2}}. That is, our definition of escaping set implies the definition of Dinesh Kumar and Sanjay Kumar.

\section{Basic Comparison of classical and holomorphic semigroup dynamics}
The following immediate relations hold for any $ f \in S $  from the definition \ref{2ab}. Indeed, it shows certain connection between classical complex dynamics and semigroup dynamics.

\begin{theorem}\label{1c}
Let $ S $  be a semigroup. Then
\begin{enumerate}  
 \item $F(S) \subset F(f)$ for all $f \in S$  and hence  $F(S)\subset \bigcap_{f\in S}F(f)$. 
\item $ J(f) \subset J(S) $ for all $f \in S$.
 \item  $I(S) \subset I(f)$ for all $f \in S$  and hence  $I(S)\subset \bigcap_{f\in S}I(f)$ in the case of transcendental semigroup $ S $.
\end{enumerate}
\end{theorem}

Hinkkanen and Martin proved the following results ({\cite[Lemma 3.1 and Corollary 3.1]{hin}}).
\begin{theorem}\label{perf}
 Let $ S $ be a rational semigroup. Then
 Julia set $ J(S) $ is perfect and $ J(S) = \overline{\bigcup_{f \in S} J(f)} $ 
 \end{theorem}

K. K. Poon proved the following results ({\cite[Theorems 4.1 and 4.2] {poo}}).
\begin{theorem}\label{perf1}
 Let $ S $ be a transcendental semigroup. Then
 Julia set $ J(S) $ is perfect and $ J(S) = \overline{\bigcup_{f \in S} J(f)} $ 
 \end{theorem}

From the  theorem \ref{1c} ((1) and (3)), we can say that the Fatou set and the escaping set may be empty.
For example, the escaping set of semigroup $ S = \langle f, g \rangle $ generated by functions $ f(z) = e^{z} $ and $ g(z) =  e^{-z}$ is empty (the particular function $ h = g \circ f^{k} \in S $ (say) is iteratively bounded at any $ z \in I(f) $).
We know that Fatou set may be empty  but escaping set is non-empty in classical complex dynamics. This is a contrast feature of escaping set in classical complex dynamics and semigroup dynamics. From the same theorem part (2), we can say that in classical and semigroup dynamics,  Julia set is non-empty.

Dinesh Kumar and Sanjay Kumar {\cite [Theorem 2.5]{kum2}} have identified the following transcendental  semigroup $S$, where $I(S)$ is an empty set.
\begin{theorem}\label{e}
The transcendental semigroup $S = \langle f_{1},\;f_{2}\rangle$  generated by two functions $f_{1}$ and $ f_{2} $ from  respectively two parameter family  $\{e^{-z+\gamma}+c\;  \text{where}\;  \gamma, c\in \mathbb{C} \; \text{and}\;  Re(\gamma)<0, \; Re(c)\geq 1\}$ and $\{e^{z+\mu}+d, \; \text{where}\;  \mu, d\in \mathbb{C} \; \text{and}\; Re(\mu)<0,  \; Re(d)\leq -1\}$ of functions  has empty escaping set $I(S)$ \end{theorem}
\begin{proof}
Under the conditions stated in the theorem, for any function $ f $ from the first family we have $I(f) \subset \{z \in \mathbb{C}: Re z < 0, (4k -3)\frac{\pi}{2}< Im z < (4k -1)\frac{\pi}{2}, k \in \mathbb{Z}\}$ (see for instance  {\cite[Lemma 2.2]{kum2}}) and for any function $ g $ from the second family we have $I(g) \subset \{z \in \mathbb{C}: Re z > 0, (4k -1)\frac{\pi}{2}< Im z < (4k +1)\frac{\pi}{2}, k \in \mathbb{Z}\}$(see for instance  {\cite[Lemma 2.3]{kum2}}). From theorem \ref{1c}, for any $ f,\; g \in S = \langle f_{1},\;f_{2}\rangle $,  we have $ I(S)\subset I(f) \cap I(g) = \emptyset $. 
\end{proof}
There are several transcendental semigroups where escaping set $ I(S) \neq \emptyset $. The following examples of Dinesh Kumar and Sanjay Kumar {\cite[Examples 2.6 and 2.7]{kum2}} are better to mention here. 
\begin{exm}
Let $ S = \langle e^{\lambda z},  e^{s\lambda z} + 2\pi i /\lambda \rangle$ for all $ \lambda \in \mathbb{C} -\{0\} $  and $ s \in \mathbb{N} $. Then  $ I(S) =I(f) \neq \emptyset $. 
\end{exm}

\begin{exm}
Let $ S = \langle f, g \rangle $, where $ f(z) = \lambda \sin z $ ($ \lambda \in \mathbb{C} -\{0\} $) and $ g(z) = f^{n} + 2 \pi $ for all $ n \in \mathbb{N} $. Then $ I(S) =I(f) \neq \emptyset $. 
\end{exm}

If escaping set $ I(S) \neq \emptyset $, then Eremenko's  result $\partial I(f) = J(f)$ \cite{ere} of classical transcendental dynamics can be generalized to semigroup settings. The following results is due to Dinesh Kumar and Sanjay Kumar {\cite [Lemma 4.2 and Theorem 4.3]{kum2}} which yield the generalized answer in semigroup settings.
\begin{theorem}\label{3}
Let $S$ be a transcendental semigroup such that $ I(S) \neq \emptyset $. Then
\begin{enumerate}
\item $int(I(S))\subset F(S)\;\ \text{and}\;\ ext(I(S))\subset F(S) $, where $int$ and $ext$ respectively denote the interior and exterior of $I(S)$.  
\item $\partial I(S) = J(S)$, where $\partial I(S)$ denotes the boundary of $I(S)$. 
\end{enumerate}
\end{theorem}
\begin{proof}
\begin{enumerate}
\item We refer for instance lemma 4.2 of \cite{kum2}.
\item The facts $int(I(S))\subset F(S)\;\ \text{and}\;\ ext(I(S))\subset F(S)$ yield $ J(S)\subset \partial I(S) $. The fact $ \partial I(S)\subset J(S) $ is obvious.
\end{enumerate}
\end{proof}
From this theorem \ref{3}, the fact $ J(S) \subset \overline{I(S)} $ follows trivially. 
If $ I(S) \neq \emptyset $, then we prove the following result which is a generalization of Eremenko's  result $I(f)\cap J(f) \neq \emptyset $ {\cite[Theorem 2]{ere}} of classical transcendental dynamics to holomorphic semigroup dynamics. 
\begin{theorem}\label{lu1}
Let $S$ be a transcendental semigroup such that $ F(S)$ has a multiply connected component. Then $I(S)\cap J(S) \neq \emptyset $
\end{theorem}
Following result of Baker {\cite[Theorem 3.1]{bak1}} is better to worth mention. 
\begin{lem}\label{lu}
Let $ f $ be a transcendental entire function and $ U $ be a multiply connected component of $ F(f) $. Then $ f^{n}(z) \to \infty $ locally uniformly on  $ U$. 
\end{lem}
\begin{proof}[Proof of the Theorem \ref{lu1}]
Suppose $ F(S) $ has a multiply connected component $ U $. Then by theorem \ref{1c} (1), $ U $ is also multiply connected component of $ F(f) $ for each $ f \in S $ and by lemma \ref{lu}, for each $ f \in S $,  $ f^{n}(z) \to \infty $ locally uniformly on  $  U $ and also that $ f^{n}(z) \to \infty $ on $ \partial U $. It follows by normality (that is, by theorem \ref{ad1}) that every sequence in $ S $ has a subsequence which diverges to $ \infty $ locally uniformly on $ U $ and $ \partial U $. This proves that $ f^{n}(z) \to \infty $ for all $ z \in U $ and $ z \in \partial U $ for each $ f \in S $ and hence by theorem \ref{1c} (3) , $ U \subset I(S) $. Since $ \partial U \subset J(f) $ for all $ f \in S $. By theorem \ref{1c} (2), $ \partial U \subset J(S) $. This proves that $I(S)\cap J(S) \neq \emptyset $. 

\end{proof}

The one of the most important result of classical complex dynamics is either $ J(f) =\mathbb{C} \; \text{or}\; \mathbb{C}_{\infty} $ or $ J(f) $ has empty interior for any holomorphic map $ f $ on $\mathbb{C} \; \text{or}\; \mathbb{C}_{\infty} $ (see {\cite[Lemma 3]{ber1}}). There are lot of examples of  transcendental entire functions and rational functions whose Julia set is entire complex plane or extended complex plane. Note that while $ J(f) = \mathbb{C} \; \text{or} \; \mathbb{C}_{\infty}$ is possible for some holomorphic map $ f $, we always have $ F(f) \neq \mathbb{C}\; \text{or} \; \mathbb{C}_{\infty} $. On the other hand, the analogous result is not hold in semigroup dynamics. 
Hinkkanen and Martin {\cite[Example-1]{hin}} provided the following example that shows that Julia set of a rational semigroup $ S $ may have non-empty interior even if $ J(S) \neq \mathbb{C}_{\infty} $. 
\begin{exm}\label{efi}
Rational semigroup $ S =\langle z^{2}, z^{2}/a \rangle $, where $ a \in \mathbb{C}, |a| >1 $ has Fatou set $ F(S) = \{z : |z| <1\; \text{or}\; |z| > |a| \}  $ and Julia set $ J(S)  = \{z : 1 \leq  |z| \leq  |a| \} $.
\end{exm}

Let $ U $ be a component of Fatou set $ F(f) $. Then $ f(U) $ is contained in  some component $ V $ of $ F(f) $. Note that if $ f $ is rational map then $ V =f(V) $. If $ f $ is transcendental, then it is possible that $ V \neq f(U) $.  Let us recall the following result of Bergweiler and Rohde \cite{ber1} of classical complex dynamics.
\begin{theorem} \label{zhi}
If $ f $ is entire, then $ V-f(U) $ contains at  most one point which an asymptotic value of $ f $.
\end{theorem} 
 The following example of Huang Zhigang {\cite[Example 2]{huang}} shows that above result (theorem \ref{zhi}) can not be preserved for general semigroup dynamics. This is a contrast between classical complex dynamics and semigroup dynamics.

\begin{exm}
Let $ S =\langle z^{n}, az^{n} \rangle $, where $ n> 2 $ and $ |a| > 1 $. The Fatou set $ F(S) $ contains following components
$$
U = \Bigg\{\sqrt [n]{\dfrac{1}{|a|}}< |z| < \sqrt[n]{\sqrt [n-1]{\dfrac{1}{|a|}}} \Bigg\} \;\; \text{and} \;\;  V =\{|z| > 1 \}. 
$$
For a function $ f(z) = a z^{n} $ in semigroup $ S $, $ f(U) \subset V $ and $ V - f(U) $ is an unbounded domain.
\end{exm}
\begin{dfn}
Let $ S $ be a holomorphic semigroup. We define the backward orbit of any $ z \in \mathbb{C} $ (or $ \mathbb{C}_{\infty} $) by
$$
O^{-}(z) = \{w\in \mathbb{C}_{\infty}: \text{there exists} \; f \in S \; \text{such that} \; f(w) =z \}
$$
and the exceptional set of $ S $ is defined by
$$
E(S) = \{z \in \mathbb{C}_{_{\infty}}: O^{-}(z) \; \text{is finite}\}
$$
\end{dfn}
Note that if $ S $ finitely generated rational semigroup, then $ E(S) \subset F(S) $, otherwise we can not assert it. For example ({\cite[example 1]{huang}}), semigroup $ S =\langle f_{m} \rangle $, where $ f_{m}(z) =a^{m}z^{n}, m\in \mathbb{N}, n\geq 2 $ and $ |a| > 1 $,  is an infinitely generated polynomial semigroup. Then, $ E(S) =\{0, \infty \} $. It is easy to see that 0 is a limit point of $ J(f_{m}) = \{|z| = |a|^{\frac{-m}{n-1}}\} $, and hence $ 0 \in J(S) $.
So, in the case of finitely generated rational semigroup $ S $, we always have $ E(S) \subset F(S) \subset F(f) $ for any $ f \in S $. Hence $ E(S) $ contains at most two points. However, if   $ S $ finitely generated transcendental semigroup, then we can not assert $ E(S) \subset F(S) $ in general because for a transcendental function, it is difficult to determine whether Fatou exceptional value belongs Fatou set or Julia set. For example, 0 is the Fatou exceptional value of $ f(z) =e^{\lambda z} $. It is known in classical complex dynamics that $ 0 \in J(f) $ if $ \lambda > 1/e $ and $ 0 \in F(f) $ if $ \lambda < 1/e $. Poon and Yang \cite{poo1} gave the following characterization whether a Fatou exceptional value belongs to Fatou set or Julia set.
\begin{theorem}
Let $ f $ is transcendental entire function. If $ F(f) $ has no unbounded component, then Fatou exceptional value always belongs to Julia set.
\end{theorem}
So, in the case of finitely generated transcendental semigroup $ S $, if $ E(f) \subset F(f)  $ for all $ f \in S $, then we can say $ E(S) \subset F(S) \subset F(f) $ for any $ f \in S $. Hence $ E(S) $ contains at most one point. 
This fact is a generalization of classical complex dynamics to semigroup dynamics and so it is a nice connection between these two types of dynamics. 
Huang Zhigang {\cite[Proposition 1]{huang}} proved the following result which also shows a connection between classical complex dynamics and semigroup dynamics.
\begin{theorem}
If $ z\notin E(S) $, then $ J(S) \subseteq \overline{O^{-}(z)} $. 
\end{theorem}

\end{document}